\documentclass[11pt,reqno]{amsart}

\usepackage[
  letterpaper,
  left=1.0in,
  right=1.0in,
  top=1.0in,
  bottom=1.0in
]{geometry}

\usepackage[T1]{fontenc}
\usepackage{lmodern}
\usepackage{amsmath,amssymb}
\usepackage{microtype}
\usepackage[hidelinks]{hyperref}
\hypersetup{
  pdftitle={Atiyah's Minkowski Space Conjecture Fails for Every n >= 3},
  pdfauthor={Ziran Liu},
  pdfsubject={Counterexamples to Atiyah's Minkowski space conjecture for every n >= 3},
  pdfkeywords={Atiyah configuration of points, Minkowski space, inertial worldline, binary form, determinant, counterexample}
}

\allowdisplaybreaks[2]

\newtheorem{theorem}{Theorem}[section]
\newtheorem{proposition}[theorem]{Proposition}
\newtheorem{lemma}[theorem]{Lemma}
\newtheorem{corollary}[theorem]{Corollary}
\newtheorem{conjecture}[theorem]{Conjecture}
\theoremstyle{remark}
\newtheorem{remark}[theorem]{Remark}

\newcommand{\M}{\mathbb M}
\newcommand{\C}{\mathbb C}
\newcommand{\R}{\mathbb R}
\newcommand{\PP}{\mathbb P}
\newcommand{\ip}[2]{\langle #1,#2\rangle}

\title[Minkowski Space Conjecture]{Atiyah's Minkowski Space Conjecture Fails for Every $n\ge3$}
\author[Z. Liu]{Ziran Liu}
\address{Shanghai Institute for Mathematics and Interdisciplinary Sciences (SIMIS), Shanghai 200433, China}
\address{
Research Institute of Intelligent Complex Systems, Fudan University, Shanghai 200433, China}
\email{zliu@simis.cn}
\date{}
\keywords{Atiyah's configuration of points, Minkowski space conjecture, celestial sphere, inertial worldline, binary form, counterexample, normalized determinant}
\subjclass[2020]{Primary 51P05; Secondary 83A05, 15A03, 15A15}

\begin{document}
\raggedbottom

\begin{abstract}
Atiyah's Minkowski-space version of the configuration-of-points construction assigns to an admissible marked configuration of $n$ worldlines a collection of $n$ binary forms of degree $n-1$, whose roots
are the ordered retarded celestial directions. He conjectured that these forms are always linearly independent.  

We disprove this conjecture for every $n\ge3$.
For $n=3$, an explicit planar one-parameter family yields a real coefficient determinant with exactly one simple zero in a specified interval.  At this parameter, all six ordered celestial roots are distinct and the coefficient matrix has rank exactly two.  A null-translation construction then multiplies the first three forms by a common factor and produces counterexamples for every $n>3$.  Consequently, within the class of complete pairwise disjoint timelike affine lines, universal independence holds at $n=2$ and fails for every $n\ge3$; for each of the counterexamples, the normalized Atiyah--Sutcliffe determinant is defined and vanishes.
\end{abstract}

\maketitle

\section{Introduction}\label{sec:intro}

\subsection{Atiyah's Minkowski-space conjecture}

The configuration-of-points problem rests on a precise dimension match.  Let
\[
 W_n=\operatorname{Sym}^{n-1}\bigl((\C^2)^*\bigr),
 \qquad \dim_{\C}W_n=n.
\]
For each label $i$, Atiyah's construction multiplies the $n-1$ linear
factors determined by the ordered directions from $i$ to the remaining
labels and thereby produces a binary form $\beta_i\in W_n$.  Thus one
obtains exactly $n$ vectors in an $n$-dimensional vector space.  Their linear
independence is the nondegeneracy condition that turns the projective classes
$[\beta_1],\ldots,[\beta_n]$ into an ordered projective frame.  After fixing a Hermitian metric, one may choose unit representatives and
place them in the columns of an invertible matrix.  The unitary polar factor
of that matrix determines an ordered orthogonal frame, hence a complete
flag.  Changing the unit representatives only changes the phases of the
columns and therefore leaves the corresponding point of $U(n)/T^n$
unchanged.  Berry and Robbins's transported-spin construction motivated the
configuration-to-flag problem \cite{BerryRobbins}.  Atiyah then formulated
the associated equivariant flag-map problem, gave an explicit solution of
its bare existence version, and proposed the binary-form ansatz used here
\cite{AtiyahClassical,Atiyah2001}.

Atiyah also formulated a retarded Minkowski version of the same ansatz.
He introduced the construction in \cite[\S6]{Atiyah2001} and stated the
conjecture explicitly in \cite[Conjecture~1.6.1]{Atiyah2010}.  In the
notation used here, Atiyah's original conjecture is the following.

\begin{conjecture}[Atiyah's Minkowski space conjecture]
\label{conj:atiyah}
Let $n\ge2$.  Let $\xi_1,\ldots,\xi_n\subset\M^{3,1}$ be $n$
nonintersecting worldlines, and let $x_1,\ldots,x_n$ be $n$ distinct events
with $x_i\in\xi_i$ for every $i$.  For each $i\ne j$, let $u_{ij}$ be the
point in the celestial sphere of $x_i$ at which the past light cone at
$x_i$ meets $\xi_j$.  Identify these celestial spheres by parallel
translation and form the polynomials $\beta_i$ from the roots $u_{ij}$,
$j\ne i$.  Then $\beta_1,\ldots,\beta_n$ are linearly independent over
$\C$.
\end{conjecture}

The conjecture is stated in the context of the preceding retarded
construction and therefore presupposes that the indicated celestial points
are defined.  For the proofs below, we work in a smaller, geometrically
well-posed subclass.  Namely, let $\xi_1,\ldots,\xi_n$ be pairwise disjoint
future-directed timelike worldlines, choose a marked event $x_i\in\xi_i$ on
each worldline, and suppose that, for every $i\ne j$, the past light cone of
$x_i$ meets $\xi_j$ in a unique point $y_{ij}$.  We call such marked data
\emph{admissible}.  Because Minkowski space is flat, parallel translation
identifies the celestial spheres based at the different marked events with
one common sky, exactly as in Atiyah's formulation
\cite[\S1.6]{Atiyah2010}.  The past-null ray
\[
 u_{ij}=[y_{ij}-x_i]_+,
 \qquad [v]_+:=\{\lambda v:\lambda>0\},
\]
then determines, after one common projective identification of this sky with
$\PP^1(\C)$, a linear factor $\ell_{ij}$, and the $i$th form is
\[
 \beta_i=\prod_{j\ne i}\ell_{ij}\in W_n.
\]
Section~\ref{sec:retarded} gives the causal construction, lift conventions,
and coordinate invariance in full detail.

The main theorem below refutes Conjecture~\ref{conj:atiyah}.  Indeed,
our counterexample consists of complete, pairwise disjoint, future-directed
timelike affine lines, and every ordered retarded intersection exists
uniquely.  It therefore satisfies Atiyah's original geometric input while
the associated polynomials are linearly dependent.  We construct the counterexamples by restricting attention
to the admissible timelike subclass, which serves only to make every retarded datum unambiguous and to show that the failure already occurs for uniform
subluminal motion; of course a counterexample in this subclass is automatically a counterexample to the original conjecture.

Atiyah singled out the case $n=3$ as a prove-or-disprove challenge and
identified uniform motion as a natural version of the question
\cite[\S1.6]{Atiyah2010}.  The purpose of this paper is to give a negative
answer in that very class and then to propagate the failure to every
particle number $n\ge3$.

\subsection{Previous results and the remaining Minkowski case}

The literature separates naturally into the theory concerning the flag-map problem, positive
results for the Euclidean and hyperbolic branches, and extensions of the polynomial formalism.  We review these developments before isolating the
uniform-motion Minkowski problem left open by Atiyah.

\medskip\noindent\textbf{Foundations and the flag-map problem.}
The transported-spin construction of Berry and Robbins supplied the physical and geometric motivation \cite{BerryRobbins}.  Atiyah formulated
the resulting equivariant flag-map problem, gave an explicit solution of its
bare existence version, and proposed the more rigid binary-form construction
considered here in \cite{AtiyahClassical} and \cite{Atiyah2001}.  Atiyah
and Bielawski subsequently proved the existence of smooth $W$-equivariant
flag maps for arbitrary compact Lie groups by using Nahm's equations
\cite{AtiyahBielawski}.  That existence theorem does not, however, prove
the nondegeneracy of Atiyah's specific binary forms.  Atiyah introduced the
forms and their coefficient determinant in \cite{AtiyahClassical} and
\cite{Atiyah2001}.  Atiyah and Sutcliffe then developed the
scale-invariant normalized determinant systematically, formulated the
stronger Conjectures~II and~III, supplied extensive numerical evidence, and
computed the associated energy-minimizing configurations for up to $32$
particles \cite{AtiyahSutcliffe}.

\medskip\noindent\textbf{Positive results in the Euclidean branch.}
The Euclidean linear-independence conjecture is elementary for $n\le3$,
and Atiyah also proved it for collinear configurations
\cite{AtiyahClassical,Atiyah2001}.  Eastwood
and Norbury proved it for every configuration of four points
\cite{EastwoodNorbury}.  {\DJ}okovi\'c proved the conjecture for two special
types of configurations, obtaining the stronger Atiyah--Sutcliffe
conclusion for one of them \cite{DjokovicSpecial}; he also treated the
nonplanar family in which some points lie on a line and the remaining points
form a regular polygon in a perpendicular plane whose centroid lies on that
line \cite{DjokovicDihedral}.  For almost-collinear configurations, Svrtan
and Urbiha proposed symmetric-function conjectures implying
Conjectures~II and~III and carried out computations through $n=6$; for
$n=4$ they also verified Conjectures~II and~III for several infinite
families of tetrahedra \cite{SvrtanUrbihaAlmost}.  Their later work verified Conjectures~II and~III
for parallelograms, cyclic quadrilaterals, and further infinite tetrahedral
families, proposed stronger four-point statements, and carried the
multi-Schur computations through $n=9$ \cite{SvrtanUrbihaStrong}.  Mazur
and Petrenko proved Conjecture~II for regular polygons and convex
quadrilaterals and Conjecture~III for inscribed quadrilaterals
\cite{MazurPetrenko}.  Bou Khuzam and Johnson gave a computer-aided proof of
Conjectures~II and~III for all Euclidean four-point configurations
\cite{BouKhuzamJohnson}, while Malkoun later gave a new proof of four-point
linear independence by showing that the associated Gram matrix is positive
definite \cite{MalkounFourPoints}.

\medskip\noindent\textbf{Hyperbolic, Lie-theoretic, and topological extensions.}
For four points in hyperbolic space, Malkoun proved linear independence for
noncoplanar configurations and for configurations in which one point lies
in the hyperbolic convex hull of the other three
\cite{MalkounHyperbolic}.  In the coplanar case, he proved the stronger
bound $|D|\ge1$ for convex quadrilaterals and also obtained the nonvanishing
result attributed there to Y.~Zhang and J.~Ma for nonconvex quadrilaterals
\cite{MalkounCoplanar}.

Beyond the unitary problem, Malkoun constructed an explicit symplectic
candidate and proved its linear-independence conjecture for $n=2$
\cite{MalkounSymplectic}; he then gave analogous explicit constructions for
the remaining classical groups \cite{MalkounClassical}.  He also introduced
a multivariable extension of the polynomial problem
\cite{MalkounSeveralVariables}, proved that the $U(2m)$ conjectures imply
the corresponding $Sp(m)$ conjectures through root-system domination
\cite{MalkounRootSystems}, formulated a weight-theoretic rank conjecture
\cite{MalkounWeights}, and derived a global expansion formula for the
Atiyah--Sutcliffe determinant \cite{MalkounDeterminant}.  His later graph extension replaces the complete graph by an arbitrary
finite simple graph; for the complete graph, its first two conjectures
specialize to Atiyah--Sutcliffe Conjectures~I and~II
\cite{MalkounGraphs}.  Finally, Guerra and Salvatore showed that suitable
strong complex and planar real forms of the Euclidean conjectures would
produce, respectively, $E_3$- and $E_2$-algebra structures on disjoint
unions of unordered full flag manifolds \cite{GuerraSalvatore}.

\medskip\noindent\textbf{The remaining uniform-motion problem.}
The Euclidean and hyperbolic branches impose reciprocity relations on
oppositely ordered roots; general retarded worldlines do not.  Atiyah's 2001
cyclic construction produced dependent forms along straight spatial tracks,
but the required velocities were nonuniform.  He therefore asked whether
the example could be modified to use straight worldlines in Minkowski space,
or whether linear independence might instead hold for all such worldlines
\cite[\S6, p.~1387]{Atiyah2001}.  In the 2010 lectures he recalled that the
earlier proposed realization involved faster-than-light motion and again
singled out uniform motion with $n=3$ as a prove-or-disprove challenge
\cite[\S1.6]{Atiyah2010}.  To the author's knowledge, the literature cited above contains no
published resolution of the uniformly moving timelike problem before the
present work.  The results below answer it negatively.  They do not contradict the
Euclidean or hyperbolic results just reviewed: the counterexample is neither static nor formed by inertial rays
issuing from a common event.

\subsection{The explicit family and the main results}
\subsubsection{A Family of Three Worldlines}
The counterexample is planar.  We work in
$\M^{2,1}=\{z=0\}\subset\M^{3,1}$ with metric
$-dt^2+dx^2+dy^2$.  This inclusion preserves the causal relations and the
retarded intersections used below.  Consequently, a counterexample in
$\M^{2,1}$ is already a counterexample in $\M^{3,1}$.

For a real parameter $T$, consider the complete affine worldlines
\begin{align}
 \xi_1(r)&=(r,0,0),\label{eq:line1}\\
 \xi_2(r)&=\left(-\frac23+r,-\frac14+\frac r2,
                         \frac15-\frac r{10}\right),\label{eq:line2}\\
 \xi_3(r)&=\left(T+r,1+\frac{4r}{5},
                         \frac18-\frac r5\right),\label{eq:line3}
\end{align}
with marked events $x_i=\xi_i(0)$.  The direction vectors and spatial
tracks are independent of $T$; only the temporal placement of the third
line changes.  Set
\[
 I=\left[\frac34,\frac9{10}\right],\qquad
 T_- =\frac{82950}{10^5},\qquad
 T_+ =\frac{82951}{10^5},\qquad J=[T_-,T_+].
\]

\subsubsection{Main Theorems}
\begin{samepage}
\begin{theorem}[The three-worldline counterexample]\label{thm:main}
There is a unique parameter $T_*\in J$ for which the three Atiyah quadratics associated with \eqref{eq:line1}--\eqref{eq:line3} are linearly
dependent.  For every $T\in I$, the three
parametrized lines are complete and pairwise disjoint, their tangent vectors
are future directed and timelike, and each ordered pair has a unique
retarded intersection.  At $T=T_*$:
\begin{enumerate}
\item[(i)] the real coefficient determinant has a simple zero;
\item[(ii)] the six ordered celestial roots are pairwise distinct; and
\item[(iii)] the coefficient matrix has rank exactly two.
\end{enumerate}
Consequently, the associated real coefficient-matrix path meets the smooth
rank-two stratum of the determinantal hypersurface
$\{\det=0\}\subset M_3(\R)$ transversely.  In particular, the timelike formulation above, and hence Atiyah's printed
Minkowski-space conjecture, is false already for three observers moving
uniformly with strictly subluminal velocities.
\end{theorem}
\end{samepage}

A counterexample for $n=3$ refutes Conjecture~\ref{conj:atiyah} as a
universal statement, but it does not by itself refute the assertion at each
prescribed $N>3$: adjoining worldlines changes both the number of forms and
their degree.  The next theorem supplies a separate geometric construction
at every particle number.

\begin{theorem}[Counterexamples for every $N\ge3$]
\label{thm:all-n-intro}
For every integer $N\ge3$, there are $N$ pairwise disjoint complete affine
worldlines in $\M^{2,1}\subset\M^{3,1}$, with future-directed timelike
direction vectors and one marked event on each line, such that every
ordered pair has a unique retarded intersection and the associated $N$
binary forms of degree $N-1$ are linearly dependent.
\end{theorem}

\begin{remark}

The two theorems give more than the failure of a single universally
quantified assertion.  Corollary~\ref{cor:inertial-dichotomy} shows that,
within the class of complete pairwise disjoint timelike affine lines, universal
independence holds exactly for $n=2$.  Corollary~\ref{cor:all-n-normalized}
shows that the normalized Atiyah--Sutcliffe determinant is defined and
vanishes on every configuration constructed here.  The three-worldline
core has six distinct roots; the stabilization for $N>3$ intentionally
introduces repeated common roots.  Atiyah's formulation starts with nonintersecting worldlines; marked
events chosen on different worldlines are therefore automatically distinct.
Neither the construction nor the printed conjecture requires all ordered
celestial roots to be pairwise distinct; see \cite[\S6]{Atiyah2001} and
\cite[Conjecture~1.6.1]{Atiyah2010}.  Requiring all ordered roots to remain
distinct for every $N$ would define a stronger problem.

\end{remark}

\begin{remark}
No common-time or spacelike-slice condition appears in Atiyah's printed
Conjecture~1.6.1 \cite{Atiyah2010}, so the marked events used here are
legitimate.  They are not simultaneous in any
inertial frame; imposing synchronization would define a different,
stronger problem.  Section~\ref{sec:geometry} records this and the other
boundaries of the conclusion precisely.
    
\end{remark}

\subsection{Main idea of the construction}

The proof of the three-line theorem is a real determinantal wall crossing.
For each $T\in I$, Proposition~\ref{prop:admissible} gives six admissible
ordered observation--emitter pairs.  Lemma~\ref{lem:retarded} solves each
retarded null equation explicitly, and the planar Cayley coordinate of
Section~\ref{sec:retarded} represents the six celestial roots by nonzero
real lifts
\[
 H_{ij}(T)\in\R^2\setminus\{0\},\qquad i\ne j.
\]
The two roots seen by observer $i$ determine the coefficient row of one
binary quadratic.  Lemma~\ref{lem:bracket} then converts the determinant of
the three coefficient rows into the six-bracket expression
\begin{equation}\label{eq:intro-delta}
\begin{aligned}
 \Delta_H(T)={}&[H_{12},H_{23}][H_{13},H_{32}][H_{21},H_{31}]\\
 &+[H_{12},H_{31}][H_{13},H_{21}][H_{23},H_{32}].
\end{aligned}
\end{equation}
This identity is the algebraic bridge between the six retarded directions
and linear dependence of the three forms.

The exact formulas in Section~\ref{sec:lifts} reduce every entry of
\eqref{eq:intro-delta} to rational expressions in the three variable
radicals $\sqrt{P_1(T)},\sqrt{P_2(T)},\sqrt{P_3(T)}$ and the three fixed
radicals $\sqrt{11},\sqrt{41},\sqrt{65}$.  Proposition~\ref{prop:crossing},
proved in Appendix~\ref{sec:interval-proof} by exact interval estimates, gives
\begin{equation}\label{eq:intro-crossing-estimates}
 \Delta_H(T_-)>0>\Delta_H(T_+),\qquad
 \Delta_H'(T)<0\quad(T\in J).
\end{equation}
The endpoint signs give a zero, and the derivative bound makes it unique in
$J$ and simple.  The same proposition gives the fifteen nonvanishing inequalities
\begin{equation}\label{eq:intro-root-separation}
 [H_{ij}(T),H_{kl}(T)]\ne0
 \quad\bigl(T\in J,\ (i,j)\ne(k,l)\bigr).
\end{equation}
Thus the six roots remain pairwise distinct at the crossing.  Since the
coefficient matrix is singular but no two quadratic rows are proportional,
its rank is exactly two.  At a rank-two matrix the differential of the
determinant is nonzero; the additional inequality $\Delta_H'(T_*)\ne0$ in
\eqref{eq:intro-crossing-estimates} therefore gives transversality to the
smooth rank-two stratum of the real determinantal hypersurface.  In
particular, the dependence is not produced by a collision, a lightlike
limit, a repeated root, or a higher-order contact: varying only the time
offset of the third line moves the retarded data through a genuine
rank-two wall.

The proof for $N>3$ uses a different mechanism.  At $T=T_*$, choose a
nonzero relation
\begin{equation}\label{eq:core-relation-intro}
 \lambda_1\beta_1^{(3)}+
 \lambda_2\beta_2^{(3)}+
 \lambda_3\beta_3^{(3)}=0.
\end{equation}
Section~\ref{sec:stabilization} constructs a Lorentzian two-plane
$\Pi=\operatorname{span}\{U_\star,\omega\}$, where $U_\star$ is future
timelike and $\omega$ is past null, and adjoins pairwise disjoint lines
\[
 \xi_k(r)=rU_\star+c_k\omega,\qquad 4\le k\le N.
\]
For each old observer $i\le3$ and each new emitter $k>3$, the retarded
displacement is a positive multiple of the same past-null vector $\omega$.
Hence every new emitter contributes the same factor $L_\star$ to each of
the first three forms, and
\begin{equation}\label{eq:proof-chain-all-n}
 \beta_i^{(N)}=L_\star^{N-3}\beta_i^{(3)},
 \qquad i=1,2,3.
\end{equation}
Multiplying \eqref{eq:core-relation-intro} by $L_\star^{N-3}$ gives a
nontrivial relation among the $N$ forms.  This common-factor identity is
therefore the geometric bridge from the nondegenerate three-line core to
counterexamples at every particle number.

\subsection{Notation and organization}

The first index in $u_{ij}$ is always the observer and the second the
emitter.  Thus $x_i$ is the marked observation event, $y_{ij}\in\xi_j$ is
the retarded emission event, and $q_{ij}=x_i-y_{ij}$ is future null.  We
reserve $T$ for the one-parameter three-line family, $T_*$ for its unique
crossing in $J$, and $N$ for the particle number in the stabilization.

Section~\ref{sec:retarded} develops the retarded celestial forms and fixes a
single projective coordinate on the planar common sky.
Section~\ref{sec:family} verifies the timelike and pairwise-disjoint nature
of the explicit family.  Section~\ref{sec:lifts} proves the six-bracket
identity and derives the six exact root lifts.  Section~\ref{sec:crossing}
proves Theorem~\ref{thm:main} from Proposition~\ref{prop:crossing}.
Section~\ref{sec:stabilization} proves Theorem~\ref{thm:all-n-intro} by the
null-translation mechanism.  Section~\ref{sec:geometry} extracts the
projective balance, normalized determinant, dichotomy, and exact scope.
Appendix~\ref{sec:interval-proof} gives the exact rational interval proof
of Proposition~\ref{prop:crossing}.

\section{Retarded celestial forms}\label{sec:retarded}

This section makes the construction in Conjecture~\ref{conj:atiyah}
precise and proves the two causal facts needed later: a complete timelike
affine emitter has a unique retarded point, and reciprocal retarded rays of
two disjoint affine lines are distinct.

\subsection{The common sky and the binary forms}

We first fix the causal and projective conventions.  Let
\[
 \M^{3,1}=(\R^4,\eta),\qquad
 \eta=-dt^2+dx^2+dy^2+dz^2,
\]
and write $\ip{\cdot}{\cdot}$ for the corresponding bilinear form.  A
nonzero vector is timelike, null, or spacelike according as its squared norm
is negative, zero, or positive; timelike and null vectors are causal.  A
causal vector is future directed when its time component is positive and past
directed when its time component is negative.  The incoming celestial
sphere is the projectivized past null cone
\begin{equation}\label{eq:celestial-quotient}
 \mathcal S^-=
 \{v\ne0:\ip{v}{v}=0,\ v^0<0\}/\R_{>0}.
\end{equation}
Flat parallel translation identifies the null cones based at different
events with this single common sky, which is precisely the identification
used in Atiyah's Minkowski construction \cite[\S1.6]{Atiyah2010}.

By a future-directed timelike worldline we mean the image of a $C^1$
embedding whose tangent is everywhere future directed and timelike.  Let
$\xi_1,\ldots,\xi_n$ be pairwise disjoint such worldlines and choose
marked events $x_i\in\xi_i$.  We call these data \emph{admissible} if,
for every $i\ne j$, the past light cone of $x_i$ meets $\xi_j$ in exactly
one point $y_{ij}$.  This condition makes Atiyah's retarded construction
single-valued; it is automatic for the complete timelike affine lines used
below.  With $[v]_+:=\{\lambda v:\lambda>0\}$, the incoming celestial
root is
\begin{equation}\label{eq:intro-root}
 u_{ij}=[y_{ij}-x_i]_+\in\mathcal S^-.
\end{equation}
Fix one projective identification
$\Phi:\mathcal S^-\to\PP^1(\C)$.  If
$h_{ij}=(a_{ij},b_{ij})\in\C^2\setminus\{0\}$ lifts $\Phi(u_{ij})$, set
\begin{equation}\label{eq:retarded-forms}
 \ell_{ij}(Z_0,Z_1)=b_{ij}Z_0-a_{ij}Z_1,
 \qquad
 \beta_i=\prod_{j\ne i}\ell_{ij}
 \in\operatorname{Sym}^{n-1}((\C^2)^*).
\end{equation}
Changing a lift rescales the corresponding form $\beta_i$ by a nonzero
scalar.  A common orientation-preserving projective change of coordinate acts on
every form through the same invertible symmetric-power transformation.  An
orientation-reversing change is the composition of such a transformation
with simultaneous complex conjugation.  Since both operations are bijective
on the coefficient space and are applied simultaneously to all forms, they
preserve both linear dependence and linear independence.

Conjecture~\ref{conj:atiyah} asserts that the forms in
\eqref{eq:retarded-forms} are linearly independent for every admissible
marked configuration.  For $n=3$ the forms are binary quadratics.  Let $M$ be their coefficient
matrix in the ordered basis
\[
 (Z_0^2,Z_0Z_1,Z_1^2).
\]
Then $\beta_1,\beta_2,\beta_3$ are linearly dependent if and only if
$\det M=0$.

\subsection{Retarded intersections with a complete timelike affine line}

We next compute the retarded root determined by one observation event and
one timelike affine emitter.  For an observation event $x$ and an emitter
$\xi(s)$, put $q(s)=x-\xi(s)$.  At the retarded point, $q(s)$ is future
null.
\begin{lemma}[Retarded point on a timelike affine line]\label{lem:retarded}
Let $\xi(s)=x_0+sU$, where $U$ is future-directed timelike, and put
$\kappa^2=-\ip{U}{U}>0$.  If $x\notin\xi(\R)$, set
\[
 d=x-x_0,\qquad A=\ip{d}{U},\qquad
 R=\sqrt{A^2+\kappa^2\ip{d}{d}}.
\]
Then $R>0$, and the past light cone of $x$ meets $\xi$ at exactly one point,
namely $\xi(s^-)$ with
\begin{equation}\label{eq:retarded}
  s^-=\frac{-A-R}{\kappa^2}.
\end{equation}
The vector
\begin{equation}\label{eq:qdef}
 q=x-\xi(s^-)=d+\frac{A+R}{\kappa^2}U
\end{equation}
is future directed and null.
\end{lemma}

\begin{proof}
The null equation is
\[
  0=\ip{d-sU}{d-sU}=\ip{d}{d}-2As-\kappa^2s^2,
\]
whose roots are $(-A\pm R)/\kappa^2$.  Choose a proper orthochronous Lorentz
transformation sending $U$ to $\kappa(1,\mathbf0)$, and write the transformed
displacement as $d=(\tau,\mathbf r)$.  Then
$A=-\kappa\tau$, $R=\kappa|\mathbf r|$, and the two roots are
\[
 s^-=(\tau-|\mathbf r|)/\kappa,\qquad
 s^+=(\tau+|\mathbf r|)/\kappa.
\]
At $s^-$ one has
$q=d-s^-\kappa(1,\mathbf0)=(|\mathbf r|,\mathbf r)$, which is future
null.  At the other root,
$\xi(s^+)-x=(|\mathbf r|,-\mathbf r)$ is future null, so $\xi(s^+)$ lies
on the future light cone of $x$.  Since $x$ does not lie on the emitter,
$|\mathbf r|>0$.  The two algebraic roots therefore give one retarded and
one advanced intersection.  In particular, $R>0$, and
\eqref{eq:retarded}--\eqref{eq:qdef} follow.  Transforming back preserves
nullity and time orientation.
\end{proof}

The proof of linear dependence needs only the retarded roots.  The normalized
determinant also requires reciprocal roots to be distinct; the following
lemma supplies that fact for every affine configuration used here.

\begin{lemma}[Reciprocal retarded rays are distinct]
\label{lem:reciprocal-distinct}
Let $\xi_i$ and $\xi_j$ be disjoint complete timelike affine lines, with
marked events $x_i\in\xi_i$ and $x_j\in\xi_j$.  Then their two retarded
directions satisfy $u_{ij}\ne u_{ji}$.
\end{lemma}

\begin{proof}
Write
\[
 \xi_i(r)=x_i+rU_i,\qquad \xi_j(s)=x_j+sU_j,
\]
with $U_i,U_j$ future timelike.  Suppose, to the contrary, that the two
retarded directions determine the same past-null ray $[w]_+$.  For some
$a,b>0$ and some $r,s\in\R$, the two retarded emission equations would be
\[
 x_j+sU_j=x_i+aw,
 \qquad
 x_i+rU_i=x_j+bw.
\]
Adding them gives
\begin{equation}\label{eq:reciprocal-contradiction}
 rU_i+sU_j=(a+b)w.
\end{equation}
If $U_i$ and $U_j$ are linearly independent, then
$w\in\operatorname{span}\{U_i,U_j\}$ by
\eqref{eq:reciprocal-contradiction}, and the first emission equation gives
$x_j-x_i\in\operatorname{span}\{U_i,U_j\}$.  Writing
$x_j-x_i=\alpha U_i+\beta U_j$ gives the common point
$x_i+\alpha U_i=x_j-\beta U_j$, contrary to disjointness.  If $U_i$ and $U_j$ are linearly dependent,
then $U_j=\lambda U_i$ for some $\lambda>0$, because both vectors are future
directed.  The left side of
\eqref{eq:reciprocal-contradiction} is therefore
$(r+\lambda s)U_i$, which is either zero or timelike, whereas the right
side is nonzero and null.  This is again impossible.
\end{proof}

\subsection{A common projective coordinate on the planar sky}

We next pass from the retarded displacement to its incoming point on the
common celestial sphere.  For an ordered pair $(i,j)$, apply
Lemma~\ref{lem:retarded} with $x=x_i$ and $\xi=\xi_j$, and denote the
resulting future null vector by $q_{ij}$.  Atiyah's root is the opposite
past-directed ray
\[
 u_{ij}=[\xi_j(s^-_{ij})-x_i]_+=[-q_{ij}]_+\in\mathcal S^-.
\]
Intersecting this ray with the section $t=-1$ gives the incoming sky
direction
\begin{equation}\label{eq:sky}
 n_{ij}=-\frac{(q^x_{ij},q^y_{ij},q^z_{ij})}{q^0_{ij}}\in S^2.
\end{equation}
For a static emitter at spatial position $\mathbf x_j$, formula
\eqref{eq:sky} gives
$n_{ij}=(\mathbf x_j-\mathbf x_i)/
\lvert\mathbf x_j-\mathbf x_i\rvert$, exactly the Euclidean direction
in \cite[Eq.~(1.1)]{Atiyah2010}.  Thus \eqref{eq:sky} uses the same time
orientation as Atiyah.  As in
Atiyah's definition, affine parallel translation identifies the null cones
at the different events; no observer-dependent rest-frame aberration is
applied.

The counterexample is planar, so $q^z_{ij}=0$ and all six directions lie on
the equator $S^1\subset S^2$.  We therefore choose a single real
projective coordinate on that circle.  This keeps the coefficient matrix
real and makes a sign-change argument possible.  The following Cayley
coordinate covers the entire equator, including its exceptional affine
point:
\begin{equation}\label{eq:global-cayley}
 C(n_x,n_y,0)=
 \begin{cases}
  [1+n_x:n_y],&(n_x,n_y)\ne(-1,0),\\
  [0:1],&(n_x,n_y)=(-1,0).
 \end{cases}
\end{equation}
Its inverse is the globally defined formula
\begin{equation}\label{eq:global-cayley-inverse}
 C^{-1}([a:b])=
 \left(\frac{a^2-b^2}{a^2+b^2},
       \frac{2ab}{a^2+b^2},0\right).
\end{equation}
Thus $C:S^1\to\PP^1(\R)$ is a bijection; the second branch in
\eqref{eq:global-cayley} removes the apparent chart singularity at
$(-1,0,0)$.  To compare this real coordinate with the standard complex
stereographic coordinate, we extend $C$ once to the whole celestial sphere
rather than choosing a coordinate separately for each observer.  Let
\[
 \zeta_{\mathrm{hom}}(n_x,n_y,n_z)=
 \begin{cases}
  [n_x+in_y:1-n_z],&(n_x,n_y,n_z)\ne(0,0,1),\\
  [1:0],&(n_x,n_y,n_z)=(0,0,1),
 \end{cases}
\]
and set
\[
 G_{\mathrm C}=\begin{pmatrix}1&i\\[2pt]1&-i\end{pmatrix},\qquad
 \Phi=G_{\mathrm C}^{-1}\circ\zeta_{\mathrm{hom}}.
\]
Here matrices act projectively on homogeneous column vectors.
Lemma~\ref{lem:atiyahcoordinate} verifies that $\Phi|_{S^1}=C$ and compares
this global choice with the standard stereographic coordinate.

Whenever $q^0_{ij}-q^x_{ij}>0$, substituting
\eqref{eq:sky} into \eqref{eq:global-cayley} and multiplying both
homogeneous coordinates by $q^0$ gives the convenient lift
\begin{equation}\label{eq:lift}
 C(n_{ij})=\Phi(u_{ij})=[h^0_{ij}:h^1_{ij}],\qquad
 h_{ij}=(q^0_{ij}-q^x_{ij},-q^y_{ij}).
\end{equation}
Indeed, if $[h]=[a:b]$, then
\[
 (n_x,n_y)=\left(\frac{a^2-b^2}{a^2+b^2},
                         \frac{2ab}{a^2+b^2}\right),
\]
and, for $a=q^0-q^x$ and $b=-q^y$, the null identity
$(q^0)^2=(q^x)^2+(q^y)^2$ gives
\begin{align*}
 a^2+b^2
  &=(q^0-q^x)^2+(q^y)^2
    =2q^0(q^0-q^x),\\
 a^2-b^2
  &=(q^0-q^x)^2-(q^y)^2
    =-2q^x(q^0-q^x),\\
 2ab&=-2q^y(q^0-q^x).
\end{align*}
For a future planar null vector, $q^0-q^x\ge0$; equality forces $q^y=0$ and
corresponds to the exceptional branch of \eqref{eq:global-cayley}.  On the
complement of that point, division by the first
identity yields
\[
 (n_x,n_y)=\left(-\frac{q^x}{q^0},-\frac{q^y}{q^0}\right),
\]
which is \eqref{eq:sky}.  Thus \eqref{eq:lift} gives all six roots in one
projective coordinate.  Section~\ref{sec:lifts} verifies that these six
lifts are nonzero.  For a future planar null vector, nonvanishing is
equivalent to $q^0_{ij}-q^x_{ij}>0$: equality forces $q^y_{ij}=0$ and hence
makes both coordinates in \eqref{eq:lift} vanish.

\begin{lemma}[Comparison with standard stereographic projection]
\label{lem:atiyahcoordinate}
Let $[a:b]\in\PP^1(\R)$ represent a planar sky direction by
\eqref{eq:lift}.  Use stereographic projection from the north pole
$(0,0,1)$, with affine coordinate
$\zeta=(n_x+in_y)/(1-n_z)$.  On the equator this becomes
\begin{equation}\label{eq:cayleyexact}
  \zeta=n_x+in_y=\frac{a+ib}{a-ib}.
\end{equation}
Consequently, the passage from our six roots $h_{ij}$ to their standard
stereographic coordinates is induced by the single matrix
$G_{\mathrm C}\in\operatorname{GL}(2,\C)$ displayed above:
\[
 [a:b]\longmapsto[a+ib:a-ib].
\]
In particular, the three homogeneous quadratics are linearly dependent in
our real coordinate if and only if they are linearly dependent in the
standard stereographic coordinate.
\end{lemma}

\begin{proof}
The displayed formula for $(n_x,n_y)$ gives
\[
 n_x+in_y=\frac{a^2-b^2+2iab}{a^2+b^2}
           =\frac{a+ib}{a-ib},
\]
which proves \eqref{eq:cayleyexact}.  The determinant of $G_{\mathrm C}$ is
$-2i$.
The corresponding contragredient change of variables therefore induces an
invertible linear map on the three-dimensional vector space of binary
quadratics.  Applying the same map to all three polynomials preserves both
linear dependence and linear independence.
\end{proof}

We have now fixed the causal and projective input once and for all.  For the
explicit family, every ordered pair will therefore produce a unique root in
one common projective coordinate, and dependence may be tested by an
ordinary real coefficient determinant.

\section{The explicit three-worldline family}
\label{sec:family}

We now verify the causal properties of the family
\eqref{eq:line1}--\eqref{eq:line3}.  Only the temporal placement of the
third line varies; the velocities and spatial tracks remain fixed.  The
parameter therefore isolates the effect of retarded timing from collisions
and changes of velocity.  For $T\in I$, let $\xi_i$ be
\eqref{eq:line1}--\eqref{eq:line3} with $T$ in place of $T_*$, and let
$x_i=\xi_i(0)$.  The three marked events are distinct because their spatial
coordinates are distinct.  The direction vectors are
\begin{equation}\label{eq:old-directions}
 U_1=(1,0,0),\qquad U_2=\left(1,\frac12,-\frac1{10}\right),
 \qquad U_3=\left(1,\frac45,-\frac15\right).
\end{equation}
These vectors are future directed and
\begin{equation}\label{eq:timelike}
 \ip{U_1}{U_1}=-1,\qquad
 \ip{U_2}{U_2}=-\frac{37}{50},\qquad
 \ip{U_3}{U_3}=-\frac8{25}.
\end{equation}
The spatial speeds are $0$, $\sqrt{26}/10$, and $\sqrt{17}/5$; each is
strictly less than one.
Reparametrizing each line by proper time changes neither the underlying
worldline nor any celestial direction.

\begin{proposition}[Causal admissibility of the family]\label{prop:admissible}
For every $T\in I$, the three complete affine lines $\xi_1,\xi_2,\xi_3$
are pairwise disjoint.  Hence each of the six ordered pairs has a unique
retarded intersection.
\end{proposition}

\begin{proof}
An intersection of $\xi_1$ and $\xi_2$ would require both spatial
coordinates of $\xi_2(r)$ to vanish.  The first forces $r=1/2$, whereas
the second then equals $3/20$.  An intersection of $\xi_1$ and $\xi_3$
would force $r=-5/4$ from the first spatial coordinate, whereas the second
would then equal $3/8$.

For $\xi_2(r)$ and $\xi_3(s)$, equality of the two spatial coordinates has
the equations
\[
 -\frac14+\frac r2=1+\frac{4s}{5},
 \qquad
 \frac15-\frac r{10}=\frac18-\frac s5,
\]
or, after clearing denominators,
\[
 10r-16s=25,\qquad 4r-8s=3.
\]
The second equation gives $r=2s+3/4$; substitution in the first gives
\[
 4s+\frac{15}{2}=25,\qquad
 s=\frac{35}{8},\qquad r=\frac{19}{2}.
\]
At this unique candidate pair of parameters, the two time coordinates are
\[
 -\frac23+\frac{19}{2}=\frac{53}{6},
 \qquad T+\frac{35}{8}.
\]
Equality of the time coordinates would force
\[
 T=\frac{53}{6}-\frac{35}{8}=\frac{107}{24}>\frac9{10}.
\]
Thus the lines are pairwise disjoint,
and Lemma~\ref{lem:retarded} supplies the six unique retarded points.
\end{proof}

\section{From six retarded roots to one determinant}\label{sec:lifts}

This section performs the algebraic reduction that drives the proof.  We
first express the coefficient determinant directly in terms of the six
projective roots, then compute those roots from the retarded-intersection
formula, and finally verify that the resulting determinant is a smooth real
function throughout the parameter interval.

\subsection{The six-bracket identity}

We begin by reducing linear dependence to one explicit determinant.  For a nonzero vector
$a=(a_0,a_1)\in\C^2$, write
\[
 L_a(Z_0,Z_1)=a_0Z_1-a_1Z_0,
\]
and, for nonzero $a,b\in\C^2$, set $p_{a,b}=L_a L_b$.  Since
$L_{h_{ij}}=-\ell_{ij}$, the two signs cancel and, for $n=3$,
$p_{h_{ij},h_{ik}}$ is exactly $\beta_i$.
Its coefficient row in the ordered basis
$(Z_0^2,Z_0Z_1,Z_1^2)$ is
\begin{equation}\label{eq:coefficientrow}
 r(a,b)=\bigl(a_1b_1,-a_0b_1-a_1b_0,a_0b_0\bigr).
\end{equation}
Write $[a,b]=a_0b_1-a_1b_0$.

\begin{lemma}[Bracket identity]\label{lem:bracket}
For $a,b,c,d,e,f\in\C^2$,
\begin{equation}\label{eq:bracketidentity}
 \det\!\begin{pmatrix}r(a,b)\\r(c,d)\\r(e,f)\end{pmatrix}
 =[a,d][b,f][c,e]+[a,e][b,c][d,f].
\end{equation}
\end{lemma}

\begin{proof}
Both sides are polynomial in the six vectors.  Under a common
$G\in\operatorname{GL}(2,\C)$, the right-hand side is multiplied by
$\det(G)^3$, one factor for each bracket.  For the left-hand side,
write $Z=(Z_0,Z_1)$ and observe that
\[
 L_{Ga}(Z)=\det(G)L_a(G^{-1}Z)
\]
gives
\[
 p_{Ga,Gb}(Z)=\det(G)^2p_{a,b}(G^{-1}Z).
\]
Thus each coefficient row acquires the scalar $\det(G)^2$, while all three
rows undergo the same change of variables
$\operatorname{Sym}^2(G^{-1})$.  Their determinant is therefore multiplied
by
\[
 \det(G)^6\det\operatorname{Sym}^2(G^{-1})
 =\det(G)^6\det(G)^{-3}=\det(G)^3.
\]
On the dense set $[a,b]\ne0$, choose
$G\in\operatorname{GL}(2,\C)$ sending $a$ and $b$ to the standard basis.
Because both sides acquire the same nonzero factor, it is enough to verify
the identity after this transformation.  For $a=(1,0)$ and $b=(0,1)$ one
has $r(a,b)=(0,-1,0)$, and the determinant on the left becomes
\[
 c_1d_1e_0f_0-c_0d_0e_1f_1.
\]
The right-hand side gives the same expression.  Since both sides are
polynomial in the six vectors, equality on this dense set proves the
identity everywhere.
\end{proof}

Apply the lemma with
$(a,b,c,d,e,f)=(h_{12},h_{13},h_{21},h_{23},h_{31},h_{32})$.  The three
rows are the coefficient rows of $\beta_1,\beta_2,\beta_3$, so these
quadratics are dependent exactly when
\begin{equation}\label{eq:Delta}
 \Delta_h=[h_{12},h_{23}][h_{13},h_{32}][h_{21},h_{31}]
       +[h_{12},h_{31}][h_{13},h_{21}][h_{23},h_{32}]
\end{equation}
vanishes.  Rescaling the six lifts multiplies $\Delta_h$ by the product of
the six scale factors.  In particular, positive rescaling preserves both
its zero set and its sign.

Equation~\eqref{eq:Delta} is the central algebraic bridge in the paper:
it reduces Theorem~\ref{thm:main} to the concrete task of computing the six roots in the common Cayley coordinate and showing that
the two triple products cancel for some $T\in I$.

\subsection{Retarded parameters and exact root lifts}

All square roots below denote positive real roots.  To keep
the radicands readable, set
\begin{align}
 P_1(T)&=1088T^2-2480T+1481,\label{eq:P1}\\
 P_2(T)&=244800T^2-404400T+173129,\label{eq:P2}\\
 P_3(T)&=14976T^2-52896T+47963.\label{eq:P3}
\end{align}
Lemma~\ref{lem:retarded} uses the displacement from the emitter's base point
to the observation event.  In the present family these six vectors are
\begin{align*}
d_{12}&=\left(\frac23,\frac14,-\frac15\right),&
d_{13}&=\left(-T,-1,-\frac18\right),\\
d_{21}&=\left(-\frac23,-\frac14,\frac15\right),&
d_{23}&=\left(-T-\frac23,-\frac54,\frac3{40}\right),\\
d_{31}&=\left(T,1,\frac18\right),&
d_{32}&=\left(T+\frac23,\frac54,-\frac3{40}\right).
\end{align*}
Accordingly, let $A_{ij}=\ip{d_{ij}}{U_j}$,
$\kappa_j^2=-\ip{U_j}{U_j}$, and let $R_{ij}$ denote the positive number
$R$ from Lemma~\ref{lem:retarded}, namely
\[
 R_{ij}:=\sqrt{A_{ij}^2+\kappa_j^2\ip{d_{ij}}{d_{ij}}}.
\]
Substitution into the retarded formula gives the six scalar products
\begin{align*}
A_{12}
 &=-\frac23+\frac18+\frac1{50}=-\frac{313}{600},\\
A_{13}
 &=T-\frac45+\frac1{40}=\frac{40T-31}{40},\\
A_{21}
 &=\frac23,\\
A_{23}
 &=T+\frac23-1-\frac3{200}=\frac{600T-209}{600},\\
A_{31}
 &=-T,\\
A_{32}
 &=-T-\frac23+\frac58+\frac3{400}
   =-\frac{1200T+41}{1200}.
\end{align*}
The corresponding radicands are
\begin{align*}
R_{12}^2
 &=\left(-\frac{313}{600}\right)^2
   +\frac{37}{50}\left(-\frac49+\frac1{16}+\frac1{25}\right)
   =\frac{11}{576},\\
R_{13}^2
 &=\left(T-\frac{31}{40}\right)^2
   +\frac8{25}\left(-T^2+1+\frac1{64}\right)
   =\frac{P_1(T)}{1600},\\
R_{21}^2
 &=\left(\frac23\right)^2-\frac49+\frac1{16}+\frac1{25}
   =\frac{41}{400},\\
R_{23}^2
 &=\left(T-\frac{209}{600}\right)^2
   +\frac8{25}\left[-\left(T+\frac23\right)^2
                         +\frac{25}{16}+\frac9{1600}\right]
   =\frac{P_2(T)}{360000},\\
R_{31}^2
 &=T^2-T^2+1+\frac1{64}=\frac{65}{64},\\
R_{32}^2
 &=\left(T+\frac{41}{1200}\right)^2
   +\frac{37}{50}\left[-\left(T+\frac23\right)^2
                           +\frac{25}{16}+\frac9{1600}\right]
   =\frac{P_3(T)}{57600}.
\end{align*}

These scalar quantities determine the retarded parameters, but the
coefficient determinant requires actual projective lifts.  Put
\(
 \lambda_{ij}=(A_{ij}+R_{ij})/\kappa_j^2.
\)
The preceding identities give
\begin{align*}
\lambda_{12}&=\frac{-313+25\sqrt{11}}{444},&
\lambda_{13}&=\frac{5(40T-31+\sqrt{P_1})}{64},\\
\lambda_{21}&=\frac23+\frac{\sqrt{41}}{20},&
\lambda_{23}&=\frac{600T-209+\sqrt{P_2}}{192},\\
\lambda_{31}&=-T+\frac{\sqrt{65}}8,&
\lambda_{32}&=\frac{-1200T-41+5\sqrt{P_3}}{888}.
\end{align*}
Equations \eqref{eq:qdef} and \eqref{eq:lift} then give one formula for all
six projective lifts:
\begin{equation}\label{eq:huniform}
 h_{ij}=\left(d_{ij}^0-d_{ij}^x+
 \frac{A_{ij}+R_{ij}}{\kappa_j^2}(U_j^0-U_j^x),
 -d_{ij}^y-\frac{A_{ij}+R_{ij}}{\kappa_j^2}U_j^y
 \right).
\end{equation}
For example, consider the light seen by observer $1$ from worldline $2$.
The quantities above read
\[
 d_{12}=\left(\frac23,\frac14,-\frac15\right),\qquad
 \kappa_2^2=\frac{37}{50},\qquad
 A_{12}=-\frac{313}{600},\qquad
 R_{12}=\frac{\sqrt{11}}{24}.
\]
Hence
\[
 \lambda_{12}=\frac{A_{12}+R_{12}}{\kappa_2^2}
 =\frac{-313+25\sqrt{11}}{444}.
\]
Since $U_2=(1,1/2,-1/10)$, formula \eqref{eq:huniform} gives
\[
 h_{12}=\left(\frac5{12}+\frac{\lambda_{12}}2,
               \frac15+\frac{\lambda_{12}}{10}\right)
 =\frac1{888}\bigl(57+25\sqrt{11},5(\sqrt{11}+23)\bigr).
\]
Thus the six lifts are obtained uniformly through the chain
\[
 d\longmapsto(A,R)\longmapsto\lambda\longmapsto q\longmapsto h.
\]
Applying the same formula to the other ordered pairs gives the complete
list below.  In this display, and until the end of the section, we suppress
the argument $T$ in $P_1(T),P_2(T),P_3(T)$.  For each lift, the first line
records the direct substitution and the second its simplified form.

\begin{align}
h_{12}
 &=\left(\frac5{12}+\frac{\lambda_{12}}2,
          \frac15+\frac{\lambda_{12}}{10}\right)\notag\\
 &=\frac1{888}\bigl(57+25\sqrt{11},5(\sqrt{11}+23)\bigr),
 \label{eq:h12}\\
h_{13}
 &=\left(1-T+\frac{\lambda_{13}}5,
          \frac18+\frac{\lambda_{13}}5\right)\notag\\
 &=\frac1{64}\bigl(\sqrt{P_1}+33-24T,
                    \sqrt{P_1}+40T-23\bigr),\label{eq:h13}\\
h_{21}
 &=\left(-\frac5{12}+\lambda_{21},-\frac15\right)\notag\\
 &=\frac1{20}\bigl(5+\sqrt{41},-4\bigr),\label{eq:h21}\\
h_{23}
 &=\left(-T+\frac7{12}+\frac{\lambda_{23}}5,
          -\frac3{40}+\frac{\lambda_{23}}5\right)\notag\\
 &=\frac1{960}\bigl(\sqrt{P_2}+351-360T,
                     \sqrt{P_2}+600T-281\bigr),\label{eq:h23}\\
h_{31}
 &=\left(T-1+\lambda_{31},-\frac18\right)\notag\\
 &=\frac18\bigl(\sqrt{65}-8,-1\bigr),\label{eq:h31}\\
h_{32}
 &=\left(T-\frac7{12}+\frac{\lambda_{32}}2,
          \frac3{40}+\frac{\lambda_{32}}{10}\right)\notag\\
 &=\frac1{1776}\Bigl(576T+5\sqrt{P_3}-1077,\notag\\[-2pt]
 &\hspace{32mm}\sqrt{P_3}+125-240T\Bigr).\label{eq:h32}
\end{align}
Define the numerator lifts $H_{ij}$ by clearing the six positive
denominators in \eqref{eq:h12}--\eqref{eq:h32}.  Thus
\begin{equation}\label{eq:scales}
 (h_{12},h_{13},h_{21},h_{23},h_{31},h_{32})
 =\left(\frac{H_{12}}{888},\frac{H_{13}}{64},
 \frac{H_{21}}{20},\frac{H_{23}}{960},
 \frac{H_{31}}8,\frac{H_{32}}{1776}\right).
\end{equation}

\subsection{A smooth real determinant}

For the interval argument, we must verify that these formulas give
continuous, nonzero lifts throughout $I$.  We first check that every
radicand is positive.  The derivative bounds
\begin{align*}
 P_1'(T)&=2176T-2480\le P_1'(9/10)=-\frac{2608}{5}<0,\\
 P_3'(T)&=29952T-52896\le P_3'(9/10)=-\frac{129696}{5}<0
\end{align*}
show that $P_1$ and $P_3$ are decreasing on $I$.  Moreover,
$P_2''(T)=489600>0$ and
\[
 P_2'(T)=0\quad\Longleftrightarrow\quad T=\frac{337}{408}\in I.
\]
Their minima on $I$ are therefore
\begin{equation}\label{eq:positiveP}
 P_1(9/10)=\frac{3257}{25},\qquad
 P_2(337/408)=\frac{103968}{17},\qquad
 P_3(9/10)=\frac{312179}{25}.
\end{equation}
It remains to check that none of the six lifts vanishes.  For $H_{13}$ and
$H_{23}$, the first coordinates satisfy
\begin{align*}
 H_{13}^0&\ge \sqrt{P_1(T)}+33-24(9/10)>\frac{57}{5},\\
 H_{23}^0&\ge \sqrt{P_2(T)}+351-360(9/10)>27.
\end{align*}
The second coordinate of $H_{13}$ also satisfies
\[
 H_{13}^1=\sqrt{P_1(T)}+40T-23
 >40(3/4)-23=7.
\]
For $H_{32}$, the inequality
$111^2<312179/25$ gives
\[
 H^1_{32}\ge \sqrt{312179/25}+125-240(9/10)>20.
\]
Finally, both coordinates of $H_{12}$ are positive, while
$H_{21}^1=-4$ and $H_{31}^1=-1$.  Hence none of the six vectors vanishes.

We now compare the coefficient matrices associated with the two choices of
lifts.  Let $M_h(T)$ be the matrix whose rows are
\[
 r(h_{12},h_{13}),\qquad r(h_{21},h_{23}),\qquad
 r(h_{31},h_{32}),
\]
and define $M_H(T)$ analogously from the numerator lifts $H_{ij}$.  Put
$\Delta_h=\det M_h$ and $\Delta_H=\det M_H$.  These are smooth real
functions on $I$.  Multilinearity and \eqref{eq:scales} give
\[
 \Delta_H=(888)(64)(20)(960)(8)(1776)\,\Delta_h.
\]
The factor is positive and constant.  Thus the two determinants have the
same zeros and the same sign, and a zero is simple for one if and only if it
is simple for the other.

\section{A unique simple determinant crossing}\label{sec:crossing}

The next proposition concerns the determinant $\Delta_H=\det M_H$ and
supplies the determinant zero, monotonicity, and root separation needed to
identify that zero as a nondegenerate rank-two crossing.  It is proved in
Appendix~\ref{sec:interval-proof} by exact rational interval estimates.

\begin{proposition}[Exact crossing and separation of the six roots]
\label{prop:crossing}
With $T_-,T_+,J$ as in Section~\ref{sec:intro}, one has
\[
 \Delta_H(T_-)>0,
 \qquad
 \Delta_H(T_+)<0,
 \qquad
 \Delta_H'(T)<0\quad(T\in J).
\]
Moreover, every bracket between two distinct vectors among
\[
 H_{12}(T),H_{13}(T),H_{21}(T),H_{23}(T),H_{31}(T),H_{32}(T)
\]
is nonzero for every $T\in J$.
\end{proposition}

\begin{proof}[Proof of Theorem~\ref{thm:main}]
The strict endpoint signs in Proposition~\ref{prop:crossing} and the
intermediate value theorem give a zero $T_*\in(T_-,T_+)$.  Since
$\Delta_H'<0$ throughout $J$, the function is strictly decreasing there;
the zero is unique in $J$, and $\Delta_H'(T_*)\ne0$.  The positive constant
relating $\Delta_H$ and $\Delta_h$ gives the same conclusions for
$\Delta_h$.

The simplicity statement is independent of the chosen smooth nonzero lift
normalization.  Indeed, independently rescaling the six lifts by smooth
nonvanishing functions multiplies the determinant by their product, so the
new determinant has the form
$\widetilde\Delta(T)=g(T)\Delta_h(T)$ with $g(T)\ne0$.  At a zero of
$\Delta_h$ one therefore has
\[
 \widetilde\Delta'(T_*)=g(T_*)\Delta_h'(T_*).
\]
A fixed common projective coordinate change contributes only an additional
nonzero constant factor to the coefficient determinant and likewise
preserves simplicity.

Proposition~\ref{prop:admissible} and \eqref{eq:timelike} show that all three
worldlines are complete, pairwise disjoint, future directed, and timelike
throughout $[3/4,9/10]$; the spatial speeds computed after
\eqref{eq:timelike} are strictly subluminal.  Lemma~\ref{lem:retarded} gives
the six unique retarded intersections.  By the final assertion of
Proposition~\ref{prop:crossing}, the six projective roots at $T_*$ are
pairwise distinct.  In particular, the three two-element root multisets are
disjoint.  Proportional nonzero homogeneous quadratics have the same root
multiset, with multiplicities; hence no two rows of the coefficient matrix
are proportional.  The matrix is singular but has rank at least two, and
therefore has rank exactly two.

At a rank-two matrix $M$, the adjugate $\operatorname{adj}(M)$ is nonzero,
so $d(\det)_M(E)=\operatorname{tr}(\operatorname{adj}(M)E)$ is a nonzero
linear functional.  Thus the determinant-zero set in $M_3(\R)$ is a smooth
hypersurface near $M_H(T_*)$.  Since
\[
 \frac{d}{dT}\det M_H(T)\bigg|_{T=T_*}=\Delta_H'(T_*)\ne0,
\]
the path $T\mapsto M_H(T)$ meets that hypersurface transversely.  Finally,
the common
$\operatorname{GL}(2,\C)$ transformation in
Lemma~\ref{lem:atiyahcoordinate} preserves linear dependence and identifies
the same counterexample in the standard stereographic coordinate.
\end{proof}

\section{Null-translation stabilization at every particle number}
\label{sec:stabilization}

The three-line crossing is the essential counterexample.  To obtain a
separate counterexample at each prescribed particle number, we now build a
geometric mechanism that preserves the three-term relation while adding
worldlines.  The argument has three parts: an algebraic common-factor
lemma, a general Lorentzian realization, and an application of the general
criterion to the explicit three-line core.

\subsection{The common-factor mechanism}

To distinguish the three-worldline forms from those of an enlarged
configuration, write
\begin{equation}\label{eq:beta-superscript}
 \beta_i^{(m)}=
 \prod_{\substack{1\le j\le m\\ j\ne i}}\ell_{ij}
 \in\operatorname{Sym}^{m-1}((\C^2)^*).
\end{equation}

\begin{lemma}[Common-factor preservation]\label{lem:common-factor}
Let $N>3$.  Suppose an admissible three-worldline configuration satisfies the
nontrivial relation
\begin{equation}\label{eq:three-relation}
 \lambda_1\beta_1^{(3)}+
 \lambda_2\beta_2^{(3)}+
 \lambda_3\beta_3^{(3)}=0.
\end{equation}
If these three worldlines are contained in an admissible
$N$-worldline configuration and, for one celestial direction $u_\star$,
\[
 u_{ik}=u_\star,
 \qquad i\in\{1,2,3\},\quad k\in\{4,\ldots,N\},
\]
then the $N$ forms are linearly dependent.
\end{lemma}

\begin{proof}
Choose one nonzero linear factor $L_\star$ vanishing at
$\Phi(u_\star)$.  Retain the six old lifts and use $L_\star$ for every
occurrence of the common new root.  With these choices, the identities
\[
 \beta_i^{(N)}=L_\star^{N-3}\beta_i^{(3)},
 \qquad i=1,2,3,
\]
hold exactly.  Multiplication of \eqref{eq:three-relation} by $L_\star^{N-3}$ gives a
nontrivial relation among the $N$ forms, with zero coefficients on the new
ones.  For arbitrary lift choices, each form is merely rescaled by a nonzero
number, and the coefficients of the relation may be rescaled inversely.
Thus the conclusion is intrinsic.
\end{proof}

\subsection{A general null-translation realization}

The common-factor argument has a geometric realization that is independent
of the particular numerical example.  Translating all worldlines and marked
events by the same vector leaves every retarded displacement, and hence
every celestial root, unchanged.  We may therefore formulate the criterion
using a linear two-plane through the origin.

\begin{proposition}[Null-translation stabilization principle]
\label{prop:null-translation}
Let $N>3$, and let $\xi_i(r)=x_i+rU_i$, $i=1,2,3$, be pairwise disjoint
complete timelike affine lines, with each $U_i$ future directed, whose
associated forms are dependent.  Suppose that the marked events lie in the
two-plane
\[
 \Pi=\operatorname{span}_{\R}\{U_\star,\omega\},
\]
where $U_\star$ is future-directed timelike and $\omega$ is a nonzero
past-directed null vector, and
suppose $U_i\notin\Pi$ for $i=1,2,3$.  Write
\[
 x_i=\alpha_i U_\star+\rho_i\omega.
\]
For any distinct real numbers $c_4,\ldots,c_N$ satisfying
$c_k>\max_i\rho_i$, the additional lines
\begin{equation}\label{eq:null-translation-general}
 \xi_k(r)=rU_\star+c_k\omega,
 \qquad x_k=c_k\omega,
 \qquad 4\le k\le N,
\end{equation}
can be adjoined to the original three lines to produce an admissible
$N$-worldline configuration with linearly dependent forms.
\end{proposition}

\begin{proof}
The vectors $U_\star$ and $\omega$ are independent, because a nonzero null
vector cannot be proportional to a timelike vector.  Hence the new parallel
lines are mutually disjoint.  Each new line lies in $\Pi$.  Since
$x_i\in\Pi$ and $U_i\notin\Pi$, the old line $x_i+\R U_i$ meets $\Pi$
only at $x_i$.  In the basis $(U_\star,\omega)$, the point $x_i$ has
$\omega$-coordinate $\rho_i$, whereas every point of the $k$th new line
has $\omega$-coordinate $c_k$.  Since $c_k>\rho_i$, the point $x_i$ does
not lie on the $k$th new line.  Any intersection of an old line with a new
one would lie in $\Pi$ and hence would have to be the marked point $x_i$;
the coordinate comparison just made rules this out.  Thus no old--new
intersection occurs.  Together with the
pairwise disjointness of the old lines and of the new lines, this proves
that all $N$ lines are pairwise disjoint.  They are complete and timelike,
so Lemma~\ref{lem:retarded} gives admissibility.

For $i\le3<k$, the point
\[
 y_{ik}=\alpha_i U_\star+c_k\omega\in\xi_k
\]
satisfies
\[
 y_{ik}-x_i=(c_k-\rho_i)\omega.
\]
Because $c_k-\rho_i>0$, this is a nonzero past-null displacement.
Uniqueness in Lemma~\ref{lem:retarded} makes it the retarded displacement,
and consequently $u_{ik}=[\omega]_+$ for every $i\le3<k$.
Adjoining the new lines does not alter any of the old worldlines, marked
events, or old--old retarded intersections.  Hence the six old factors and
the original three-term relation are unchanged.  The conclusion then follows
from Lemma~\ref{lem:common-factor}.
\end{proof}

\subsection{Application to the three-line construction}

We now verify the hypotheses for the three-line core.  Fix the parameter
$T_*$ from Theorem~\ref{thm:main}.  Its marked events are
\begin{equation}\label{eq:old-events-stabilization}
 x_1=(0,0,0),\qquad
 x_2=\left(-\frac23,-\frac14,\frac15\right),\qquad
 x_3=\left(T_*,1,\frac18\right).
\end{equation}
Set
\begin{equation}\label{eq:stabilizing-U}
 U_\star=-x_2=\left(\frac23,\frac14,-\frac15\right),
 \qquad
 \kappa_\star^2=-\ip{U_\star}{U_\star}=\frac{1231}{3600}.
\end{equation}
Thus $U_\star$ is future directed and timelike.  We next choose a null
vector in the marked-event plane.  Specifically, we solve
$\ip{x_3-\sigma U_\star}{x_3-\sigma U_\star}=0$.  Set
\begin{align}
 A_\star&=\ip{x_3}{U_\star}=\frac{27-80T_*}{120},&
 B_\star&=\ip{x_3}{x_3}=\frac{65}{64}-T_*^2,\label{eq:stabilizing-AB}\\
 \Theta_*&=23616T_*^2-69120T_*+91679,&
 R_\star&=\sqrt{A_\star^2+\kappa_\star^2B_\star}
          =\frac{\sqrt{\Theta_*}}{480},\label{eq:stabilizing-R}\\
 \sigma_\star&=\frac{R_\star-A_\star}{\kappa_\star^2},&
 \omega&=x_3-\sigma_\star U_\star.\label{eq:stabilizing-omega}
\end{align}
Because $B_\star>65/64-81/100>0$, one has $R_\star>|A_\star|$ and
$\sigma_\star>0$.

\begin{lemma}[The stabilizing plane]\label{lem:stabilizing-null}
The vector $\omega$ is nonzero, null, and past directed, and
\[
 \Pi:=\operatorname{span}_{\R}\{x_2,x_3\}
     =\operatorname{span}_{\R}\{U_\star,\omega\}.
\]
Moreover, none of the three original direction vectors lies in $\Pi$.
\end{lemma}

\begin{proof}
The definition of $\sigma_\star$ and the identity
$R_\star^2=A_\star^2+\kappa_\star^2B_\star$ give
\[
 \ip{\omega}{\omega}
 =B_\star-2A_\star\sigma_\star
       -\kappa_\star^2\sigma_\star^2=0.
\]
Also
\[
 \ip{U_\star}{\omega}
 =A_\star+\kappa_\star^2\sigma_\star=R_\star>0.
\]
In the rest frame of the future timelike vector $U_\star$, this inequality
says that the time component of $\omega$ is negative.  Hence $\omega$ is
past directed, and in particular nonzero.  Since a nonzero null vector
cannot be proportional to a timelike vector, the identities
$x_2=-U_\star$ and $x_3=\sigma_\star U_\star+\omega$ prove the equality
of the two spanning planes.

For the original directions $U_i$ in \eqref{eq:old-directions}, direct
calculation gives
\begin{align*}
 \det[\,x_2\;x_3\;U_1\,]&=-\frac{37}{160},\\
 \det[\,x_2\;x_3\;U_2\,]&=\frac{36T_*-59}{480}<0,\\
 \det[\,x_2\;x_3\;U_3\,]&=\frac{88T_*-25}{800}>0.
\end{align*}
The signs use $3/4<T_*<9/10$.  Thus $U_i\notin\Pi$ for all $i$.
\end{proof}

\begin{proof}[Proof of Theorem~\ref{thm:all-n-intro}]
The case $N=3$ is Theorem~\ref{thm:main}.  Let $N>3$.  In the basis
$(U_\star,\omega)$, the three marked events are
\[
 x_1=0,
 \qquad x_2=-U_\star,
 \qquad x_3=\sigma_\star U_\star+\omega,
\]
so $(\rho_1,\rho_2,\rho_3)=(0,0,1)$.  Choose $c_k=k-2$ for
$k=4,\ldots,N$.  These numbers are distinct and satisfy
$c_k>\max_i\rho_i=1$.  Proposition~\ref{prop:null-translation} therefore
produces the required admissible $N$-worldline configuration and preserves
the three-term dependence.
\end{proof}

\begin{remark}[Genericity of the stabilized examples]
For $N>3$, the first three forms acquire the repeated common factor
$L_\star^{N-3}$.  The worldlines remain pairwise disjoint; coincident
celestial directions are permitted by Atiyah's conjecture.  A stronger
variant requiring all roots to be distinct is a different problem.  The
three-worldline core itself has six pairwise distinct roots by
Proposition~\ref{prop:crossing}.
\end{remark}

\section{Geometry of the rank defect, normalization, and discussion}\label{sec:geometry}

\subsection{Projective balance and Veronese concurrence}

For six nonzero lifts, the bracket identity \eqref{eq:Delta} has a
projective reformulation.  Whenever its second triple product is nonzero,
set
\begin{equation}\label{eq:multiratio}
 \mathcal R=
 \frac{[h_{12},h_{23}][h_{13},h_{32}][h_{21},h_{31}]}
 {[h_{12},h_{31}][h_{13},h_{21}][h_{23},h_{32}]}.
\end{equation}

\begin{proposition}[Projective balance at the crossing]
\label{prop:projective-balance}
The quantity $\mathcal R$ is invariant under independent rescaling of the
six lifts and under a simultaneous element of $\operatorname{PGL}(2,\C)$.
At the parameter $T_*$ of Theorem~\ref{thm:main},
\[
 \mathcal R(T_*)=-1.
\]
The space of linear relations among $\beta_1,\beta_2,\beta_3$ is
one-dimensional, and every nonzero relation
$c_1\beta_1+c_2\beta_2+c_3\beta_3=0$ satisfies
$c_1c_2c_3\ne0$.
\end{proposition}

\begin{proof}
Each lift occurs once in the numerator and once in the denominator of
\eqref{eq:multiratio}; a common $G\in\operatorname{GL}(2,\C)$ multiplies
every bracket by $\det G$.  Thus $\mathcal R$ is projectively invariant.
Proposition~\ref{prop:crossing} makes all six roots pairwise distinct, so
every bracket in \eqref{eq:multiratio} is nonzero.  Dividing
\eqref{eq:Delta} by the second triple product gives
$\Delta_h=0$ if and only if $\mathcal R=-1$.

The coefficient matrix has rank two by Theorem~\ref{thm:main}, so its
relation space is one-dimensional.  If one coefficient in a nonzero
relation vanished, the remaining two quadratics would be proportional,
contrary to the fact that their root multisets are pairwise disjoint.
\end{proof}

The same rank defect has a classical projective interpretation.  Use the
natural pairing of $\operatorname{Sym}^2((\R^2)^*)$ with
$\operatorname{Sym}^2(\R^2)$.  A binary quadratic $p$ then defines the line
\[
 L_p=\{[w]\in\PP(\operatorname{Sym}^2(\R^2)):\langle p,w\rangle=0\}.
\]
Consider the Veronese map
\[
 \nu_2:\PP(\R^2)\longrightarrow\PP(\operatorname{Sym}^2(\R^2)),
 \qquad [v]\longmapsto[v\odot v],
\]
where $\odot$ denotes the symmetric product.  If $p$ has two distinct roots
$[a]$ and $[b]$, then the line $L_p$ is precisely the secant through
$\nu_2([a])$ and $\nu_2([b])$.  The rank-two conclusion therefore
says that the three secants associated with the three observers are
concurrent.

At the same parameter, the coefficient-matrix path
$T\mapsto M_H(T)$ meets the smooth rank-two stratum of $\{\det=0\}$ in
$M_3(\R)$.  Since $\Delta_H'(T_*)\ne0$, the intersection is transverse.
We make no transversality claim in the full complex ambient space, where the
determinantal divisor has real codimension two.  Thus the failure is not
caused by a collision, a lightlike limit, a repeated root, or a higher-order
tangency along the displayed real path.  Within this family, it is a rank defect produced solely by retarded timing
inside the strictly timelike inertial region.

\subsection{The normalized determinant and the affine-timelike dichotomy}

For an admissible $N$-tuple whose reciprocal roots are distinct, choose
lifts $h_{ij}$ and let $C_N$ be the matrix whose $i$th row is the coefficient
row of $\beta_i^{(N)}$ in the basis
$(Z_0^{N-1},Z_0^{N-2}Z_1,\ldots,Z_1^{N-1})$.  Define
\begin{equation}\label{eq:general-normalized-D}
 D_{\mathrm{ret}}^{(N)}=
 \frac{\det C_N}
 {\displaystyle\prod_{1\le i<j\le N}[h_{ij},h_{ji}]}.
\end{equation}
The quotient in \eqref{eq:general-normalized-D} is the retarded form of
the normalized determinant developed by Atiyah and Sutcliffe.  Their
construction divides the coefficient determinant by the product of
reciprocal-pair brackets, and they apply the same normalized function to
Minkowski retarded data \cite[\S\S3,10]{AtiyahSutcliffe}.  Our displayed
basis and bracket conventions fix one representative of this normalization;
changing those global conventions can multiply the quotient by a fixed
nonzero scalar, but cannot change its zero set.

\begin{corollary}[Normalized zeros for every $N\ge3$]
\label{cor:all-n-normalized}
For every $N\ge3$ and every configuration constructed in the proof of
Theorem~\ref{thm:all-n-intro}, the quotient
\eqref{eq:general-normalized-D} is well defined, independent of the chosen
lifts, invariant under a common $\operatorname{PGL}(2,\C)$ coordinate
change, and equal to zero.
\end{corollary}

\begin{proof}
Lemma~\ref{lem:reciprocal-distinct} makes every denominator factor in
\eqref{eq:general-normalized-D} nonzero.  Each ordered lift $h_{ij}$ occurs exactly once in the product defining
$\beta_i^{(N)}$ and exactly once in the reciprocal bracket
$[h_{ij},h_{ji}]$.  Rescaling that lift therefore multiplies the
corresponding polynomial row and its unique denominator bracket by the same
scalar, so the quotient is independent of the chosen lifts.

Under a common $G\in\operatorname{GL}(2,\C)$, each degree-$(N-1)$ form
acquires the scalar $\det(G)^{N-1}$ and undergoes the common substitution
$\operatorname{Sym}^{N-1}(G^{-1})$.  Factoring the scalar from all $N$ rows
contributes $\det(G)^{N(N-1)}$ to the coefficient determinant.  Moreover,
\[
 \det\operatorname{Sym}^{N-1}(G^{-1})
   =\det(G)^{-N(N-1)/2};
\]
for diagonalizable $G$ this follows by multiplying the $N$ symmetric-power
weights, and the general case follows by density.  The numerator therefore
acquires the net factor $\det(G)^{N(N-1)/2}$.  The denominator contains
$N(N-1)/2$ brackets and acquires the same factor, so the quotient is
projectively invariant.  An anti-M\"obius change is a projective linear
change followed by simultaneous conjugation; it therefore conjugates the
quotient and, in particular, preserves its vanishing.  Finally,
Theorem~\ref{thm:all-n-intro} makes the rows of $C_N$ linearly dependent, so
$\det C_N=0$.
\end{proof}

\begin{corollary}[Dichotomy for complete timelike affine lines]
\label{cor:inertial-dichotomy}
Within the class of pairwise disjoint complete future-directed timelike
affine lines, Atiyah's universal independence assertion holds for $n=2$ and
fails for every $n\ge3$.
\end{corollary}

\begin{proof}
The failure for every $n\ge3$ is Theorem~\ref{thm:all-n-intro}.  For $n=2$,
the two forms are linear and vanish at $u_{12}$ and $u_{21}$, respectively.
Those roots are distinct by Lemma~\ref{lem:reciprocal-distinct}; hence the
two forms are not proportional and are therefore linearly independent.
\end{proof}

\subsection{What the counterexamples do and do not settle}

The counterexamples refute Atiyah's printed universal assertion and its
uniform-motion interpretation \cite[\S1.6]{Atiyah2010}.  The
three-worldline core is nondegenerate in
the following concrete sense: its six ordered roots are pairwise distinct,
its coefficient matrix has rank exactly two, and the displayed real path
crosses the smooth rank-two stratum of the determinant hypersurface
transversely.  For $N>3$, the
stabilization intentionally introduces repeated common roots, which are
permitted by Atiyah's formulation
\cite[Conjecture~1.6.1]{Atiyah2010}.  Whether counterexamples exist for
every $N$ with
all ordered roots pairwise distinct is a stronger genericity problem and is
not settled here.

No simultaneous-observation condition appears in Atiyah's printed
Conjecture~1.6.1 \cite{Atiyah2010}.  The marked events in our
construction cannot be simultaneous in any inertial frame: one has
\[
 \ip{x_2-x_1}{x_2-x_1}=-\frac{1231}{3600}<0,
\]
whereas the difference of two distinct points on a spacelike affine
hyperplane is spacelike.  Requiring synchronized observations would
therefore define a different, stronger problem.  Finally, the examples are
neither static nor formed by inertial rays from a common event.  Atiyah
identified those two specializations with the Euclidean and hyperbolic
branches, respectively; see \cite[\S6]{Atiyah2001} and
\cite[\S1.6]{Atiyah2010}.  Consequently, the
present counterexamples do not address either of those separate conjectures.

\appendix
\section{Proof of Proposition~\ref{prop:crossing}}\label{sec:interval-proof}

This appendix proves Proposition~\ref{prop:crossing}.  Every terminating
decimal below is interpreted as the corresponding rational number.  For
intervals $X=[x_-,x_+]$ and $Y=[y_-,y_+]$, we use the standard rational
operations
\begin{align*}
 X+Y&=[x_-+y_-,x_++y_+],\\
 XY&=[\min S,\max S].
\end{align*}
Here $S=\{x_-y_-,x_-y_+,x_+y_-,x_+y_+\}$.  Division by an interval
$Y$ not containing zero means multiplication by $[1/y_+,1/y_-]$.
All interval endpoints below are rational, and each square-root enclosure
follows from explicit rational inequalities of the form
\[
 L^2<q<U^2.
\]
We prove Proposition~\ref{prop:crossing} in four steps: uniform radical
bounds, endpoint signs, strict monotonicity on $J$, and separation of the
six roots.

\subsection{Step 1: Uniform radical enclosures}

We use the notation $T_-,T_+,J$ fixed in
Section~\ref{sec:intro}.  The derivatives $P_1',P_2',P_3'$ are increasing
affine functions, and
\[
 P_1'(T_+)<0,\qquad P_2'(T_-)>0,\qquad P_3'(T_+)<0.
\]
Hence $P_1$ and $P_3$ are decreasing on $J$, while $P_2$ is increasing.
The following uniform bounds hold:
\begin{align}\label{eq:uniform-radical-bounds}
 \sqrt{11}&\in[3.31662,3.31663],&
 \sqrt{P_1(T)}&\in[13.13216,13.13243],\notag\\
 \sqrt{41}&\in[6.40312,6.40313],&
 \sqrt{P_2(T)}&\in[78.22274,78.22286],\notag\\
 \sqrt{65}&\in[8.06225,8.06226],&
 \sqrt{P_3(T)}&\in[119.95844,119.95962].
\end{align}
The six enclosures are established by the following positive rational
margins:
\[
\begin{array}{c|cc}
q & q_{\min}-L^2 & U^2-q_{\max}\\ \hline
11&0.0000317756&0.0000345569\\
41&0.0000542656&0.0000737969\\
65&0.0001249375&0.0000363076\\
P_1(J)&0.0000557632&0.0002857049\\
P_2(J)&0.0001468924&0.0013700996\\
P_3(J)&0.0002301040&0.0023665444
\end{array}
\]
Here $[L,U]$ is the corresponding interval in
\eqref{eq:uniform-radical-bounds}, while $q_{\min}$ and $q_{\max}$ are the
appropriate monotone endpoint values on $J$.  Positivity of the two displayed
margins gives
\[
 L^2<q_{\min}\le q\le q_{\max}<U^2,
\]
and hence proves each enclosure in \eqref{eq:uniform-radical-bounds}.

\subsection{Step 2: The endpoint signs}

At the two endpoints $T_-$ and $T_+$, the following sharper rational enclosures hold:
\[
\begin{array}{c|cc}
 &T=T_-&T=T_+\\ \hline
\sqrt{P_1}&[13.132419,13.132420]&[13.132162,13.132163]\\
\sqrt{P_2}&[78.222740,78.222741]&[78.222851,78.222852]\\
\sqrt{P_3}&[119.959610,119.959611]&[119.958440,119.958441]
\end{array}
\]
and
\begin{align*}
 \sqrt{11}&\in[3.316624,3.316625],&
 \sqrt{41}&\in[6.403124,6.403125],\\
 \sqrt{65}&\in[8.062257,8.062258].
\end{align*}
Each enclosure follows from the strict rational inequalities
$L^2<q<U^2$, with $q=P_\nu(T_\pm)$ or $q\in\{11,41,65\}$, as appropriate.

Let
\[
 a=\sqrt{11},\qquad b=\sqrt{41},\qquad c=\sqrt{65},
 \qquad s_\nu(T)=\sqrt{P_\nu(T)}.
\]
The six brackets in \eqref{eq:Delta}, formed from the numerator lifts
$H_{ij}$, are
\begin{align*}
 B_1={}&(57+25a)(s_2+600T-281)
        -(115+5a)(s_2+351-360T),\\
 B_2={}&(s_1+33-24T)(s_3+125-240T)\\
       &\quad -(s_1+40T-23)(576T+5s_3-1077),\\
 B_3={}&-(5+b)+4(c-8),\\
 B_4={}&-(57+25a)-(115+5a)(c-8),\\
 B_5={}&-4(s_1+33-24T)-(5+b)(s_1+40T-23),\\
 B_6={}&(s_2+351-360T)(s_3+125-240T)\\
       &\quad -(s_2+600T-281)(576T+5s_3-1077).
\end{align*}
Substitution of the preceding rational enclosures into
$B_1,\ldots,B_6$ yields
\[
\begin{array}{c|cc}
 &T=T_-&T=T_+\\ \hline
B_1 &[24079.175326,24079.183624]&[24080.489444,24080.497742]\\
B_2 &[1189.410506,1189.410696]&[1189.296098,1189.296288]\\
B_3 &[-11.154097,-11.154092]&[-11.154097,-11.154092]\\
B_4 &[-148.107728,-148.107570]&[-148.107728,-148.107570]\\
B_5 &[-370.732120,-370.732080]&[-370.731762,-370.731723]\\
B_6 &[5817.982138,5817.983790]&[5817.378762,5817.380415]
\end{array}
\]
The two triple products themselves satisfy
\[
\begin{array}{c|cc}
 &B_1B_2B_3&B_4B_5B_6\\ \hline
T=T_-&[-319453768,-319453463]&[319455087,319455551]\\
T=T_+&[-319440473,-319440168]&[319421649,319422113].
\end{array}
\]
Since $\Delta_H=B_1B_2B_3+B_4B_5B_6$, adding the two intervals in each row
gives
\begin{align}\label{eq:endpoint-signs-J}
 \Delta_H(T_-)&\in[1319,2088]\subset(0,\infty),\\
 \Delta_H(T_+)&\in[-18824,-18055]\subset(-\infty,0).
\end{align}

\subsection{Step 3: Strict monotonicity on the isolating interval}

Write $s_\nu=\sqrt{P_\nu}$, so
$s_\nu'=P_\nu'/(2s_\nu)$.  From
\eqref{eq:uniform-radical-bounds} one obtains, throughout $J$,
\begin{align*}
 s_1'&\in[-25.701,-25.699],&
 s_2'&\in[11.014,11.047],\\
 s_3'&\in[-116.920,-116.916].
\end{align*}
Differentiating the three moving lifts gives
\begin{align*}
 H_{13}'&=(s_1'-24,s_1'+40),\\
 H_{23}'&=(s_2'-360,s_2'+600),\\
 H_{32}'&=(576+5s_3',s_3'-240).
\end{align*}
The constant lifts have zero derivative.  Hence
\begin{align*}
 B_1'&=[H_{12},H_{23}'],\\
 B_2'&=[H_{13}',H_{32}]+[H_{13},H_{32}'],\\
 B_5'&=[H_{13}',H_{21}],\\
 B_6'&=[H_{23}',H_{32}]+[H_{23},H_{32}']
\end{align*}
and $B_3'=B_4'=0$.  The uniform radical bounds
\eqref{eq:uniform-radical-bounds} give, throughout $J$,
\[
\begin{array}{c|c@{\qquad}c|c}
B_1&[24079,24081]&B_1'&[131406,131416]\\
B_2&[1189,1190]&B_2'&[-11449,-11447]\\
B_3&[-11.155,-11.154]&B_5'&[35,36]\\
B_4&[-148.109,-148.106]&B_6'&[-60459.78,-60442.07]\\
B_5&[-370.738,-370.726]&&\\
B_6&[5815,5821]&&
\end{array}
\]
Since $B_3'=B_4'=0$, differentiation gives
\[
 \Delta_H'
 =B_1'B_2B_3+B_1B_2'B_3+B_4B_5'B_6+B_4B_5B_6'.
\]
The displayed bounds give
\begin{align*}
 B_1'B_2B_3&\in[-1744475122,-1742720301],\\
 B_1B_2'B_3&\in[3074402819,3075471082],\\
 B_4B_5'B_6&\in[-31037130,-30143273],\\
 B_4B_5B_6'&\in[-3319824419,-3318677322].
\end{align*}
Adding these four rational intervals gives
\begin{equation}\label{eq:derivative-bound}
 \Delta_H'(T)\in
 [-2020933852,-2016069814]\subset(-\infty,0),
 \qquad T\in J.
\end{equation}

\subsection{Step 4: Separation of the six roots}

Applying the uniform bounds \eqref{eq:uniform-radical-bounds} to the
fifteen pairwise brackets yields the following strict estimates.  The row and
column labels denote the ordered
pairs indexing the lifts; for example, the entry in row $12$, column $13$
means $[H_{12},H_{13}]<-188$ throughout $J$.
\[
\begin{array}{c|rrrrr}
 &13&21&23&31&32\\ \hline
12&<-188&<-2060&>24079&<-148&>6340\\
13&&<-370&>4689&<-27&>1189\\
21&&&>3885&<-11&>525\\
23&&&&<-148&>5815\\
31&&&&&>3
\end{array}
\]
Every bracket is therefore nonzero on $J$.  The endpoint signs
\eqref{eq:endpoint-signs-J}, the derivative bound
\eqref{eq:derivative-bound}, and the fifteen nonvanishing estimates
complete the proof of Proposition~\ref{prop:crossing}.

\end{document}